\RequirePackage{fix-cm}
\documentclass[smallextended]{svjour3}       % onecolumn (second format)
\smartqed  % flush right qed marks, e.g. at end of proof

\usepackage{amsmath,amssymb,color,mathtools}
\usepackage{graphicx}
\mathtoolsset{showonlyrefs}

\newcommand{\reals}{\mathbb{R}}

\usepackage{xfrac}

\newcommand{\twonorm}[1]{\left\| #1 \right\|}

\makeatletter

\renewcommand{\paragraph}{%
  \@startsection{paragraph}{4} {\z@}
  {1ex \@plus 1ex \@minus .2ex}
  {-1ex}%
  {\bfseries}%
}

\begin{document}

\title{Linear Convergence of Cyclic SAGA%
\thanks{
%Youngsuk Park was XXX.
Ernest Ryu was supported in part by NSF grant DMS-1720237 and ONR grant N000141712162.
}
%Grants or other notes
}


\author{
Youngsuk Park \and Ernest K. Ryu
}


\institute{Y. Park \at
              Department of Electrical Engineering, Stanford University, \\
              \email{youngsuk@stanford.edu}    
           \and
           E. K. Ryu \at
           Department of Mathematical Sciences, Seoul National University\\
            %   7324 Mathematical Sciences,\\
%               UCLA\\
% Los Angeles, CA 90095\\
              \email{ernestryu@snu.ac.kr}  
}

\date{Received: date / Accepted: date}

\maketitle

\begin{abstract}
In this work, we present and analyze C-SAGA, a (deterministic) cyclic variant of SAGA. C-SAGA is an incremental gradient method that minimizes a sum of differentiable convex functions by cyclically accessing their gradients. Even though the theory of stochastic algorithms is more mature than that of cyclic counterparts in general, practitioners often prefer cyclic algorithms. We prove C-SAGA converges linearly under the standard assumptions. Then, we compare the rate of convergence with the full gradient method, (stochastic) SAGA, and incremental aggregated gradient (IAG), theoretically and experimentally. 
\keywords{Cyclic updates \and SAGA \and IAG \and Incremental methods \and Just-in-time update \and Linear convergence}
\end{abstract}
\section{Introduction}

Consider the optimization problem
\begin{equation}
\underset{x\in \mathbb{R}^d}{\text{minimize}}
\quad f(x):=\frac{1}{n}\sum^n_{i=1}f_i(x).
\tag{P}
\label{eq:mainP}
\end{equation}
where $f_i: \reals^d \rightarrow \reals$ for $i=1,\dots,n$ are 
convex and differentiable.
%strongly convex with constant $\mu >0$ %and has Lipschitz gradients with constant $L$.
%In addition, the the condition of strong convexity is common especially 
%when a quadratic regularizer is used. 
This finite sum structure commonly appears in machine learning problems minimizing empirical risk.
When $n$, the number of components,  is small, we can solve \eqref{eq:mainP} with the classical gradient method, which computes (full) gradients of $f$ every iteration.
When $n$ is large, however, the cost of evaluating a full gradient even once can be very expensive, and algorithms with lower per-iteration cost become appealing.

Recently, the class of incremental methods that solve  \eqref{eq:mainP} by computing the gradient of only one $f_i$ at each iteration has received much attention.
Such methods include stochastic gradient methods \cite{robbins:1951,shor1962,nemirovsky1983,shor1985,polyak1987,kushner2003}, its variance reduced variants \cite{leroux2012,mairal2013,johnson2013,zhang2013,shalev2103,defazio2014b:saga,defazio2014:finito,nitanda2014,xiao2014,mairal2015,lan2015,schmidt2016,shalev2016,shalev2016b}, and deterministic incremental methods \cite{bertsekas:2011,bertsekas2015,Wang2015,mokhtari2016:c-finito,gurbuzbalaban2016nips,gurbuzbalaban2016,gurbuzbalaban2016stronger,guo2017,parrilo2017:iag}.
Among these methods, variance reduced gradient methods such as SAGA \cite{defazio2014b:saga} achieve a faster rate of convergence compared to classical gradient descent under certain assumptions.

Stochastic incremental methods randomly access one $\nabla f_i$ at each iteration.
The randomness is essential to the theoretical analysis, and the theory for such stochastic incremental methods have significantly matured in the past five years.
In contrast, deterministic incremental methods deterministically choose one $\nabla f_i$ at each iteration, and their theoretical analysis is much weaker than that of the stochastic incremental methods even though they often perform well empirically.

% Like other variance reduced methods such as SAG and SVRG, 
% SAGA \cite{defazio2014b:saga} is an incremental gradient methods where 
% a component is chosen in (uniformly) random and
% only its gradient is accessed at each iteration.
% Aggregating previous gradients stored in a memory as well, 
% the (gradient-like) counterpart finally has a lower variance than 
% that of a (naive) stochastic gradient method 
% and SAGA shows fast linear convergence rates in expectation.

Nevertheless, many practitioners prefer cyclic methods, which access the components (deterministically) cyclically. 
Iterations of cyclic methods can be faster due to systemic reasons such as cache locality, and it is often important to guarantee all components are accessed once within each epoch  (a pass through $n$ components) especially when only a few epochs are used.
Whether practitioners should use cyclic methods is an interesting question in its own right.
Regardless, analyzing the theoretical strength or weakness of the cyclic methods practitioners use is of practical interest.

%  
% because the total number of epochs $K$ can be small.

% Even for a large epoch number $K$, 
% the deterministic incremental methods, moreover,  
% have weaker convergence results than their ramdonized versions in general.

In this paper, we present a cyclic variant of SAGA, which we call C-SAGA.
We prove it converges linearly under the standard assumptions of strong convexity and smoothness.
We theoretically and empirically compare the rate of convergence with other deterministic and stochastic incremental methods.

% \youngsuk{
% Moreover, our analysis is applicable to any order 
% as long as all components are accessed once within an epoch.
% % In addition, the rate is independent of $n$
% % whereas the rates of existing incremental aggregated-type methods 
% % decreases linearly on $n$, 
% % which make them in less favor of large $n$ scenarios.
% % and hold deterministically (not in expectation). 
% }

% {\ernest{I don't think we need to discuss worst-case analysis.}}

% because some adversarial order can be chosen and 
% the their analysis implies worst-case scenarios. 
% Therefore, under the worst-case scenario ...
% Therefore, the analysis in expectation (under uniform sampling schemes) 
% might deviate from the actual analysis of cyclic SAGA,
% and thus the worst case analysis becomes more desirable.

% {\youngsuk{worst cases in ordering, epoch number $K$ is small, components number $n$ is large}}

% depends on the number of components $n$,
% which is not in favor of big data scenario $n >> 1$.....

\paragraph{Main method.}
C-SAGA has the same algorithmic structure as SAGA, except that the choice of $\nabla f_i$ is cyclic, not random.
Define $[k]_n=\mathrm{mod}(k,n)+1$. 
We can think of $[\cdot]_n$ as the modulo operator with 1-based indexing.
We can write C-SAGA as
\begin{align*}
%\text{for }&k=1,\dots,\\
x^{k+1}&=x^k-\gamma
\left(
\nabla f_{[k]_n}(x^k)-\nabla f_{[k]_n}(x^{k-n})
+\frac{1}{n}\sum^{n}_{i=1}
\nabla f_{[k-i]_n}(x^{k-i})
\right)
\end{align*}
where $x^0,x^{-1},x^{-2}\dots,x^{-n}$ are some starting points.
For computational efficiency, implementations of C-SAGA should store the $n$ most recent gradients. 
This way, the $x^{k+1}$-update will compute one new gradient $\nabla f_{[k]_n}(x^{k})$ and use past gradients stored in memory.

\paragraph{Prior work.}
The original analysis of SAGA by Defazio et al.\ \cite{defazio2014b:saga} relies crucially on the random index selection and is not adaptable to the deterministic cyclic setup.
%and whether the cyclic variant of SAGA converges was, to the best of our knowledge, not known.
Ying et al.\  \cite{8472787} analyzed a random permutation variant of SAGA using a similar but different Lyapunov function and established the rate $1-\frac{1}{108\kappa^2}+\mathcal{O}(1/\kappa^4)$ for the random permutation setup.

Stochastic incremental methods such as Finito, MISO, SVRG, SAG, SAGA, and SDCA
\cite{leroux2012,mairal2013,johnson2013,zhang2013,shalev2103,defazio2014b:saga,defazio2014:finito,nitanda2014,xiao2014,mairal2015,lan2015,schmidt2016,shalev2016,shalev2016b}
have been an intense and fruitful area of research in recent years.
These methods achieve a faster rate of convergence compared to classical  gradient descent.

Classical stochastic and deterministic incremental gradient methods require diminishing stepsize for convergence \cite{bertsekas:2011,agarwal2012}.
The diminishing stepsize limits the rate of convergence to a sublinear rate, usually $\mathcal{O}(1/k)$ or slower, even under the assumptions of strong convexity and smoothness.

IAG, which can be viewed as a cyclic variant of SAG \cite{schmidt2016}, was the first deterministic incremental method to achieve a linear rate of convergence.
Blatt et al.\ \cite{blatt2007:iag} first proved linear convergence for quadratic components, and G\"urb\"uzbalaban et al.\ \cite{parrilo2017:iag} recently proved the linear  (epoch-by-epoch) rate with factor
$1 - \mathcal{O}(1/n\kappa^2)$ for the general strongly convex case.
Mokhtari et al.\ \cite{mokhtari2016:c-finito} presented DIAG, an incremental method that can be viewed as a cyclic variant of Finito, and proved a linear (epoch-by-epoch) rate with factor $1 - \mathcal{O}(1/\kappa)$.
Other IAG-type methods include \cite{Tseng2014,mokhtari2016:c-finito,guo2017,gurbuzbalaban2016nips,gurbuzbalaban2016,gurbuzbalaban2016stronger}.

\section{Convergence analysis}
\label{s:proof}
We now analyze the convergence
of C-SAGA.
The main result of this work is stated as Theorem~\ref{thm:main},
which we prove in several steps.
\begin{theorem}
\label{thm:main}
Assume $f_i$ is $\mu$\nobreakdash-strongly convex
and $\nabla f_i$ is $L$-Lipschitz 
continuous for all $i=1,\dots,n$.
Define $V^k = \twonorm{x^k - x^\star}^2 + \frac{1}{n} \sum_{j=1}^n \twonorm{x^k - x^{k-j}}^2$. For $0< \gamma \leq \frac{\mu}{65\sqrt{n(n+1)}L^2}$, $V^k$ converges linearly. In particular, for  $\gamma = \frac{\mu}{130\sqrt{n(n+1)}L^2}$, 
\[
V^{k+n} \leq 
\left(1 - \frac{1}{368\kappa^2}
\right)
V^k,
\]
where $\kappa=L/\mu$.
\end{theorem}

The goal of the analysis
is to establish a rate for
$\twonorm{x^k - x^\star}^2\rightarrow 0$.
However, directly using 
$\twonorm{x^k - x^\star}^2$
as a Lyapunov function in the analysis
seems difficult, as it does not
monotonically decrease.
We therefore use the
Lyapunov function $V^k$
%$V^k = \twonorm{x^k - x^\star}^2 + \frac{1}{n} \sum_{i=1}^n \twonorm{x^k - \phi_i^k}^2$,
which does monotonically decrease.
%Theorem~\ref{thm:main} states that $V^k$ converges Q-linearly, and by definition this implies $\twonorm{x^k - x^\star}^2$ converges R-linearly.

% Under this choice of the Lyapunov function,
% the technical difficulty in the proof is
% to bound the gradient error $\Delta^k = \left(x^{k+1}- (x^k - n\gamma \nabla f(x^k)) \right)/n\gamma$ 
% every epoch.
% This term essentially have some inherent recursive structure over iteration $i$, 
% and thus the careful decomposition of some related terms is necessary 
% in order to avoid the geometric growth of error bound over $n$.

% \paragraph{Remarks.}
% In fact, the Lyapunov function and the proof techniques 
% used to analyze SAGA 
% such as the fact that the expected gradient error of $\Delta^k$ is zero
% seems not be applicable when analyzing
% C-SAGA.
% Moreover, there might be a significant discrepency 
% between the convergence rate in expectation of SAGA 
% and the (worst-case) rate of C-SAGA, 
% under the circumstance that the epoch number $K$ is small.
% \youngsuk{the worst case analysis in deterministic setting
% vs. the convergence rate in expectation in stochastir setting}

IAG and C-SAGA are similar incremental and aggregated-type algorithms with essentially the same computational cost per iteration.
However, the rate for IAG shown in G\"urb\"uzbalaban et al.\ \cite{parrilo2017:iag} is $1 - \mathcal{O}(1/n\kappa^2)$ while the rate for C-SAGA shown in Theorem~\ref{thm:main} is $1 - \mathcal{O}(1/\kappa^2)$, which is better.
%which is better especially under a large number of components $n$.
%We do not know the reason for this discrepancy.
One explanation for this discrepancy is that IAG is indeed slower (in the worst case) than C-SAGA, and our computational experiments support this possibility.
Another possibility is that the analysis for IAG is not tight.
% The Lyapunov function $V^k$ is different from the Lyapunov 
% functions used in the anlayses of SAG or IAG
% \cite{parrilo2017:iag,blatt2007:iag,defazio2014b:saga}.

In \cite{mokhtari2016:c-finito}, Mokhtari et al.\ showed an even better rate of $1 - \mathcal{O}(1/\kappa)$ for their method DIAG.
However, DIAG, which can be viewed as a cyclic variant of Finito, cannot take advantage of ``just-in-time'' updates, a technique applicable to SAG, SAGA, IAG, and C-SAGA that reduces the computational cost when the gradients are sparse \cite{schmidt2016,defazio2014b:saga}.
For certain subtle reasons, this technique does not work with Finito, as Defazio et al.\ acknowledge in their work in saying ``We do not recommend the usage of Finito when gradients are sparse.'' \cite{defazio2014:finito} and, by extension, to DIAG.
So when the gradients are sparse, an iteration of C-SAGA with just-in-time updates runs faster than an iteration of DIAG, and C-SAGA can still be theoretically competitive with DIAG.

Finally, we point out that Theorem~\ref{thm:main} combined with the inequality $f(x^k)-f(x^\star)\le (L/2)\|x^k-x^\star\|^2$, where $x^\star$ is a minimizer and $L$ is the Lipschitz constant of $\nabla f$, leads to a linear rate of convergence on function values.
\begin{corollary}
In the setup of Theorem~\ref{thm:main}, if $x^\star$ is the solution, then 
\[
f(x^{kn})-f(x^\star)\le 
\frac{L}{2}\left(1 - \frac{1}{368\kappa^2}
\right)^k
V^0
\]
 for $k=0,1,\dots$.
\end{corollary}
More specifically, applying the result $n$ times with $V^0,\dots,V^{n-1}$ establishes R-linear convergence for the whole sequence $f(x^{k})-f(x^\star)$ for $k=0,1,\dots$.

\subsection{Main proof}
We now present the main theoretical contribution, the proof. We first start by stating a few inequalities that are well-known.
Throughout this section, we write $[k]$ in place  $[k]_n$ for the sake of brevity.
Throughout this section, assume $f_i$ is $\mu$\nobreakdash-strongly convex and $\nabla f_i$ is $L$-Lipschitz continuous for all $i=1,\dots,n$.

% \begin{proof}
% By Jensen's inequality, we have
% \[
% \twonorm{\frac{1}{n}v_1+\dots+\frac{1}{n}v_n}^2
% \le 
% \frac{1}{n}\left(\twonorm{v_1}^2+\dots+\twonorm{v_n}^2\right)
% \]
% multiplying both sides by $n$ gives us the desired result.
% \qed\end{proof}
% \begin{proof}
% The function 
% \[
% \beta\|a\|_M^2+\beta^{-1}\|b\|_M^2-2\langle a,b\rangle_M
% =
% \begin{bmatrix}
% a^T&b^T
% \end{bmatrix}
% \begin{bmatrix}
% \beta M&-M\\
% -M&2\beta^{-1}M
% \end{bmatrix}
% \begin{bmatrix}
% a\\b\\
% \end{bmatrix}
% \]
% is a convex function minimized at $(a,b)=(0,0)$.
% so
% \[
% \beta\|a\|_M^2+\beta^{-1}\|b\|_M^2-2\langle a,b\rangle_M\ge 0
% \]
% for all $(a,b)$. Rearranging the terms completes the proof.
% \qed\end{proof}
Let $v_1,\dots,v_n\in \reals^d$. Then
\begin{equation}
\left\|\frac{1}{n} \sum_{i=1}^n v_i\right\|^2\le \frac{1}{n} \sum_{i=1}^n \|v_i\|^2,
\label{eq:triangle-like}
\end{equation}
which follows from Jensen's inequality.
For any $a, b \in \reals^d$
and $\beta>0$, we have
\begin{equation}
\|a+b\|^2\le (1+\beta)\|a\|^2+(1+\beta^{-1})\|b\|^2,
\label{eq:beta}
\end{equation}
which can be found in references like \cite{defazio2014b:saga}.
Young's inequality states
\begin{equation}
2\langle a, b \rangle \le  \varepsilon\|a\|^2+\varepsilon^{-1}\|b\|^2,
\label{eq:young}
\end{equation}
for any $\varepsilon>0$.
If $f$ is $\mu$-strongly convex and $\nabla f$ is $L$-Lipschitz,
\begin{equation}
\langle x-y,\nabla f(x)-\nabla f(y)\rangle
\ge \frac{1}{L+\mu}\|\nabla f(x)-\nabla f(y)\|_2^2
+\frac{L\mu}{L+\mu}\|x - y\|_2^2
\label{eq:lem:cocoercivity2}
\end{equation}
for any $x,y$ \cite[Theorem 2.1.5]{nesterov2014}.

\begin{lemma}
\label{lem:e-val}
Let $\sigma_0=1$, $\tau_0=0$, and
\[
\begin{bmatrix}
\sigma_{k+1}\\
\tau_{k+1}
\end{bmatrix}
=
\begin{bmatrix}
1+c_1&1\\
c_2& 1
\end{bmatrix}
\begin{bmatrix}
\sigma_{k}\\
\tau_{k}
\end{bmatrix}
\]
for $k=0,1,\dots$.
Assume $c_1,c_2\ge 0$ and $1+c_1\ge c_2$. Then 
\[
\begin{bmatrix}
\sigma_{k}\\
\tau_{k}
\end{bmatrix}
\leq
\lambda^{k}_1
\begin{bmatrix}
1+\frac{c_1}{2\sqrt{c_1^2+4c_2}}\\
\frac{2c_2}{2\sqrt{c_1^2+4c_2}}
\end{bmatrix},
\]
where $\lambda_1 = 1+\frac{c_1+\sqrt{c_1^2+4c_2}}{2}$ .
\end{lemma}
% \begin{proof}
% This follows from the eigenvector
% decomposition of the 
% linear system defining
% $(\sigma_k,\tau_k)\mapsto (\sigma_{k+1},\tau_{k+1})$
% and working out simple algebra.
% \qed
% \end{proof}
\begin{proof}
The eigenvalues and eigenvectors of $A=
\begin{bmatrix}
1+c_1&1\\
c_2& 1
\end{bmatrix}$ are
\[
\lambda_1=1+\frac{c_1+\sqrt{c_1^2+4c_2}}{2}
\qquad
\lambda_2=1+\frac{c_1-\sqrt{c_1^2+4c_2}}{2}
\]
and 
\[
v_1
=
\begin{bmatrix}
c_1+\sqrt{c_1^2+4c_2}\\
2c_2
\end{bmatrix}
\qquad
v_2
=
\begin{bmatrix}
c_1-\sqrt{c_1^2+4c_2}\\
2c_2
\end{bmatrix}.
\]
It is simple to verify  
$\lambda_1,\lambda_2\ge 0$ under $1 + c_1  \geq c_2$.
From the following decompostion
\[
\begin{bmatrix}
\sigma_0\\
\tau_0
\end{bmatrix}
=
\begin{bmatrix}
1\\
0
\end{bmatrix}
=
\frac{1}{2\sqrt{c_1^2+4c_2}}v_1
-
\frac{1}{2\sqrt{c_1^2+4c_2}}v_2,
\]
we have
\begin{align*}
    \begin{bmatrix}
    \sigma_{k}\\
    \tau_{k}
    \end{bmatrix}
    &=
    A^k
    \begin{bmatrix}
    \sigma_0\\
    \tau_0
    \end{bmatrix}
    =
    \frac{1}{2\sqrt{c_1^2+4c_2}} 
    \left( \lambda_1^k v_1
    -
    \lambda_2^k v_2
    \right)
    \\
    &=
    (\lambda_1^k - \lambda_2^k)
    \begin{bmatrix}
    \frac{c_1}{2\sqrt{c_1^2+4c_2}}\\
    \frac{2c_2}{2\sqrt{c_1^2+4c_2}}
    \end{bmatrix}
    + 
    (\lambda_1^k + \lambda_2^k)
    \begin{bmatrix}
    \frac12\\
    0
    \end{bmatrix}
    % =
    % \begin{bmatrix}
    % \lambda_1^k 
    % \left(\frac{c_1}{2\sqrt{c_1^2+4c_2}} + \frac12\right)
    % -
    % \lambda_2^k 
    % \left( \frac{c_1}{2\sqrt{c_1^2+4c_2}} - \frac12\right)\\
    % (\lambda_1^k - \lambda_2^k)\frac{2c_2}{2\sqrt{c_1^2+4c_2}}
    % \end{bmatrix}
    \\
    &\leq
    \lambda_1^k 
    \begin{bmatrix}
    1+\frac{c_1}{2\sqrt{c_1^2+4c_2}} \\
    \frac{2c_2}{2\sqrt{c_1^2+4c_2}}
    \end{bmatrix}
\end{align*}
\qed
\end{proof}

Define
\[
\bar{g}^k=\frac{1}{n}\sum^{n}_{i=1}\nabla
f_{[k-i]}(x^{k-i}),
\]
so we can write
$
x^{k+1}=x^k-\gamma
\left(
\nabla f_{[k]}(x^k)-\nabla f_{[k]}(x^{k-n})
+\bar{g}^k
\right)
$.
\begin{lemma}
\label{lem:g-bound}
Let $x\in \reals^d$ be arbitrary.  Then
\begin{align*}
\|\bar{g}^k\|_2^2
    &\le
\frac{2L^2}{n}\left(
n  \twonorm{
x-x^\star
    }^2 +
\sum^{n}_{i=1}
    \twonorm{
x^{k-i}-x
    }^2 \right).
\end{align*}
\end{lemma}
\begin{proof}
\begin{align*}
\|
\bar{g}^k
\|_2^2
&=
    \twonorm{
    \frac{1}{n}\sum^{n}_{i=1}\nabla f_{[k-i]}(x^{k-i})-\nabla f_{[k-i]}(x^\star)
    }^2 \\
&\le
\frac{1}{n}\sum^{n}_{i=1}
    \twonorm{
    \nabla f_{[k-i]}(x^{k-i})-\nabla f_{[k-i]}(x^\star)
    }^2 \\
&\le
\frac{L^2}{n}\sum^{n}_{i=1}
    \twonorm{
x^{k-i}-x^\star
    }^2 
\le
\frac{2L^2}{n}\left(
n  \twonorm{
x-x^\star
    }^2 +
\sum^{n}_{i=1}
    \twonorm{
x^{k-i}-x
    }^2 \right).
\end{align*}
Here the first equality follows from the optimality condition,
the first inequality follows from
\eqref{eq:triangle-like},
the second inequality follows from $L$-Lipschitz gradient assumption,
and
the third inequality follows from \eqref{eq:beta}.
\qed\end{proof}

Define
% \[
% \sigma_0=1\qquad \tau_0=0
% \]
% and 
\[
\sigma_1=
1+ \beta + 4(1+\beta^{-1})\gamma^2 L^2 ,
\qquad
\tau_1=(4/n)(1+\beta^{-1})\gamma^2L^2,
\]
and 
\[
\sigma_j = \sigma_{j-1}\sigma_1+\tau_{j-1},
\qquad
\tau_j = \sigma_{j-1}\tau_1+\tau_{j-1},
\]
for $j=2,3,\dots$. Note this follows the recurrence defined in Lemma \ref{lem:e-val} when $c_1 = \beta + 4(1+\beta^{-1})\gamma^2 L^2$ and $c_2 = (4/n)(1+\beta^{-1})\gamma^2L^2$ are specified.
% where
% \[
% \sigma_1=
% 1+ \beta + 4(1+\beta^{-1})\gamma^2 L^2 ,
% \qquad
% \tau_1=(4/n)(1+\beta^{-1})\gamma^2L^2
% \]

\begin{lemma}
\label{lem:recursive}
Let $x\in \reals^d$. Then
\[
    \twonorm{x^{k} -x}^2 \le
    \sigma_1
     \twonorm{
    x^{k-1} -x
    }^2 
 +
n\tau_1
     \twonorm{
    x^{\star} -x
    }^2 +
n\tau_1
     \twonorm{
    x^{k-1-n} -x
    }^2
    + 
    \tau_1\sum^{n+1}_{i=2}\twonorm{x^{k-i}-x}^2.
\]

\end{lemma}
\begin{proof}
\begin{align*}
&\twonorm{x^{k}-x}^2 
    =  
    \twonorm{x^{k-1}-x-\gamma
\left(
\nabla f_{[k-1]}(x^{k-1})-\nabla f_{[k-1]}(x^{k-1-n})
+\bar{g}^{k-1}
\right)
    }^2 \\
    &
    \le (1+\beta)
    \twonorm{
    x^{k-1} -x
    }^2 
    + 
    2(1+\beta^{-1})\gamma^2
    \twonorm{
    \nabla f_{[k-1]}(x^{k-1}) - \nabla f_{[k-1]}(x^{k-1-n}) 
    }^2 
    + 
    2(1+\beta^{-1})\gamma^2
    \twonorm{
    \bar g^{k-1}
    }^2 
    \\
    &
    \le (1+\beta)
    \twonorm{
    x^{k-1} -x
    }^2 
    + 
    2(1+\beta^{-1})\gamma^2L^2
    \twonorm{x^{k-1} -x^{k-1-n} 
    }^2 
    + 
    2(1+\beta^{-1})\gamma^2
    \twonorm{
    \bar g^{k-1}
    }^2 
    \\
    &
    \le (1+\beta)
    \twonorm{
    x^{k-1} -x
    }^2 
    + 
    4(1+\beta^{-1})\gamma^2L^2
    \twonorm{x^{k-1} - x
    }^2 
    + 
    4(1+\beta^{-1})\gamma^2L^2
    \twonorm{x- x^{k-1-n}
    }^2 
    \\
&    + 
    2(1+\beta^{-1})\gamma^2
\frac{2L^2}{n}\left(
n  \twonorm{
x-x^\star
    }^2 +
\sum^{n}_{i=1}
    \twonorm{
x^{k-1-i}-x
    }^2 \right)
    \\
    &
    = \sigma_1
    \twonorm{
    x^{k-1} -x
    }^2 
    + 
n\tau_1
    \twonorm{x- x^{k-1-n}
    }^2 
    +
    n\tau_1
  \twonorm{
x-x^\star}^2
    + 
   \tau_1
\sum^{n}_{i=1}
    \twonorm{
x^{k-1-i}-x
    }^2.
\end{align*}
Here, the first inequality uses \eqref{eq:beta} twice, one with $\beta >0$ and the other with $\beta=1$.  The second inequality uses $L$-Lipschitz and 
the third inequality uses \eqref{eq:beta} and Lemma~\ref{lem:g-bound}.
% In the last equality, we shift the summation index.
\qed\end{proof}

% \begin{lemma}
% Then
% \begin{align*}
% \|y^k_i-x^k\|^2\le&
% n\tau_1\sigma_{n-1}\left(
%     \twonorm{
%      x^k - x^\star
%     }^2
%     +
% 2\sum^n_{j=1}
% \|\phi^k_j-x^k\|^2
%     \right)
% \end{align*}
% \end{lemma}

% Let 
% % {\color{red}
% \[
% \gamma=\frac{c}{L\sqrt{n(n+1)}}
% \]
% % }
% where $c>0$.

(For Lemmas~\ref{lem:ybound1}, \ref{lem:Lp2_bound}, and \ref{lem:bound_delta}, we use the same $\gamma$ and $c$.)
\begin{lemma}
\label{lem:ybound1}
Let $\gamma=c/(L\sqrt{n(n+1)})$, where $c>0$.
Let $1\le j\le n$.
Then
\begin{align*}
\twonorm{x^{k+j}-x^{k}}^2    
\le
\frac{33c^2}{n}
\exp\left(8c^2\right)
\sum^{n}_{i=1}\twonorm{x^{k-i}-x^{k}}^2+
16.5c^2
\exp\left(8c^2\right)
\frac{j}{n}
\twonorm{
    x^{\star} -x^{k}
    }^2.
\end{align*}
\end{lemma}

\begin{proof}
Apply Lemma~\ref{lem:recursive} with $x=x^{k-j}$
recursively to get
\begin{align*}
&\twonorm{x^k-x^{k-j}}^2    
\le
    \sigma_1
     \twonorm{
    x^{k-1} -x^{k-j}
    }^2 
 +
n\tau_1
     \twonorm{
    x^{\star} -x^{k-j}
    }^2 +
n\tau_1
     \twonorm{
    x^{k-1-n} -x^{k-j}
    }^2
    + 
    \tau_1\sum^{n+1}_{i=2}\twonorm{x^{k-i}-x^{k-j}}^2\\
&\le
\sigma_2
     \twonorm{
    x^{k-2} -x^{k-j}
    }^2 
+
\tau_1\sum^j_{i=3}\twonorm{x^{k-i}-x^{k-j}}^2
+
\tau_1\sum^{n+1}_{i=j+1}\twonorm{x^{k-i}-x^{k-j}}^2\\
&\quad+
\sigma_1\tau_1\sum^j_{i=3}\twonorm{x^{k-i}-x^{k-j}}^2
+
\sigma_1\tau_1\sum^{n+2}_{i=j+1}\twonorm{x^{k-i}-x^{k-j}}^2\\
&\quad+
n\tau_1(1+\sigma_1)\twonorm{
    x^{\star} -x^{k-j}
    }^2+
n\tau_1
\left(
     \twonorm{
    x^{k-1-n} -x^{k-j}
    }^2+
    \sigma_1
     \twonorm{
    x^{k-2-n} -x^{k-j}
    }^2
\right)\\
&=
\sigma_2
     \twonorm{
    x^{k-2} -x^{k-j}
    }^2 
+
\tau_2\sum^j_{i=3}\twonorm{x^{k-i}-x^{k-j}}^2
+
\tau_1\sum^2_{\ell=1}\sigma_{\ell-1}
\sum^{n+\ell}_{i=j+1}\twonorm{x^{k-i}-x^{k-j}}^2\\
&\quad+
n\tau_1(\sigma_0+\sigma_1)\twonorm{
    x^{\star} -x^{k-j}
    }^2
    +
n\tau_1
\sum^2_{\ell=1}\sigma_{\ell-1}
     \twonorm{
    x^{k-\ell-n} -x^{k-j}
    }^2.
\end{align*}
We now do this recursively $j$ times to get
\begin{align*}
&\twonorm{x^k-x^{k-j}}^2    
\le
\tau_1\sum^j_{\ell=1}\sigma_{\ell-1}
\sum^{n+\ell}_{i=j+1}\twonorm{x^{k-i}-x^{k-j}}^2\\
&\quad+
n\tau_1(\sigma_0+\dots+\sigma_{j-1})\twonorm{
    x^{\star} -x^{k-j}
    }^2+
n\tau_1
\sum^j_{\ell=1}\sigma_{\ell-1}
     \twonorm{
    x^{k-\ell-n} -x^{k-j}
    }^2\\    
&\le
\tau_1n\sigma_{n-1}
\sum^{n+j}_{i=j+1}\twonorm{x^{k-i}-x^{k-j}}^2\\
&\quad+
n\tau_1j\sigma_{j-1}\twonorm{
    x^{\star} -x^{k-j}
    }^2+
n\tau_1\sigma_{n-1}
\sum^j_{\ell=1}
     \twonorm{
    x^{k-\ell-n} -x^{k-j}
    }^2\\
&\le
2n\tau_1\sigma_{n-1}
\sum^{n+j}_{i=j+1}\twonorm{x^{k-i}-x^{k-j}}^2+
n\tau_1j\sigma_{j-1}\twonorm{
    x^{\star} -x^{k-j}
    }^2.
\end{align*}
Finally, we shift the indices to get
\begin{equation}
\twonorm{x^{k+j}-x^{k}}^2    
\le
2n\tau_1\sigma_{n-1}
\sum^{n}_{i=1}\twonorm{x^{k-i}-x^{k}}^2+
n\tau_1j\sigma_{j-1}\twonorm{
    x^{\star} -x^{k}
    }^2.
    \label{eq:shift-index}
\end{equation}

Now we bound $\sigma_{n-1}$ and specify $\tau_1$
with the choice $\beta=1/n$.
Write
\[
c_1= \beta + 4(1+\beta^{-1})\gamma^2 L^2 = \frac{1}{n} + \frac{4c^2}{n},
\qquad 
c_2=(4/n)(1+\beta^{-1})\gamma^2L^2 = \frac{4c^2}{n^2}.
\]
%Therefore, we have $\tau_1 = c_2 = \frac{4c^2}{n^2}$. Then Let's apply Lemma~\ref{lem:e-val} to bound $\sigma_i$.
We use Lemma~\ref{lem:e-val} to get
\begin{align*}
\lambda_1&=
1+\frac{1}{2n}+\frac{1}{2n}
\left(
4c^2
+\sqrt{1+24c^2+16c^4}
\right)\\
&\le
1+\frac{1}{2n}+\frac{1}{2n}
\left(
4c^2
+\sqrt{(1+12c^2)^2}
\right)\\
&=
1+\frac{1}{n}+\frac{8c^2}{n}
\le
\exp\left(
\frac{1}{n}+\frac{8c^2}{n}
\right)
\end{align*}
and
\[
1+\frac{c_1}{2\sqrt{c_1^2+4c_2}}
=1+\frac{1+4c^2}{2\sqrt{(1+4c^2)^2+16c^2}}
\le 1.5.
\]
So for any $i = 1,\ldots,n$, we have
\[
\sigma_{i}
\le 
\lambda_1^{i} \left(1+\frac{c_1}{2\sqrt{c_1^2+4c_2}}\right) 
\le 
1.5 \exp\left(\frac{i}{n}+\frac{8c^2i}{n}\right)
\le 
1.5\exp(1)
\exp\left(8c^2\right).
\]
Combine this bound in the inequality above and the fact $\tau_1 = c_2$ to get the stated result.
\qed
\end{proof}

\begin{lemma}\label{lem:Lp2_bound}
Let $\gamma=c/(L\sqrt{n(n+1)})$, where $c>0$. Let $1\le j\le n$.
Then 
\[
\twonorm{x^{k+n}-x^{k+j}}^2    
\le
33
c^2(1+50c^2)\exp(16c^2)
\left(\twonorm{
    x^{\star} -x^{k}
    }^2+
    \frac{2}{n}\sum^{n}_{i=1}\twonorm{x^{k-i}-x^{k}}^2
    \right)
\]
\end{lemma}
\begin{proof}
%Using the argument of Lemma~\ref{lem:ybound1},
Again, apply Lemma~\ref{lem:recursive} recursively $n$ times to get
\begin{align*}
\twonorm{x^{k+n}-x^{k+j}}^2    
&\le
2n\tau_1\sigma_{n-1}
\sum^{n}_{i=1}\twonorm{x^{k-i}-x^{k+j}}^2+
\tau_1n^2\sigma_{n-1}\twonorm{
    x^{\star} -x^{k+j}
    }^2\\
&\le
4n\tau_1\sigma_{n-1}
\sum^{n}_{i=1}\left(
\twonorm{x^{k-i}-x^{k}}^2
+
\twonorm{x^{k}-x^{k+j}}^2
\right)
+
2\tau_1n^2\sigma_{n-1}
\left(
\twonorm{    x^{k} -x^{k+j}    }^2
+
\twonorm{    x^{\star} -x^{k}    }^2
    \right)\\
&\le
4n\tau_1\sigma_{n-1}
\sum^{n}_{i=1}\left(
\twonorm{x^{k-i}-x^{k}}^2
+
2n\tau_1\sigma_{n-1}
\sum^{n}_{i=1}\twonorm{x^{k-i}-x^{k}}^2+
n^2\tau_1\sigma_{n-1}\twonorm{
    x^{\star} -x^{k}
    }^2
\right)
\\
&+
2\tau_1n^2\sigma_{n-1}
\left(
\twonorm{    x^{\star} -x^{k}    }^2
+
2n\tau_1\sigma_{n-1}
\sum^{n}_{i=1}\twonorm{x^{k-i}-x^{k}}^2+
n^2\tau_1\sigma_{n-1}\twonorm{
    x^{\star} -x^{k}
    }^2
    \right)
    \\
&=
4\sigma_{n-1}n\tau_1
(1 + 2\sigma_{n-1}n^2\tau_1 + \sigma_{n-1}n^2\tau_1)
\sum^{n}_{i=1}\twonorm{x^{k-i}-x^{k}}^2
\\
&\quad +
2\sigma_{n-1}n^2\tau_1
(1 +  2\sigma_{n-1} n\tau_1
+
\sigma_{n-1}n^2\tau_1 
) \twonorm{
    x^{\star} -x^{k}
    }^2\\
    &\le
2\sigma_{n-1}n^2\tau_1
(1 +  2\sigma_{n-1} n\tau_1
+
\sigma_{n-1}n^2\tau_1 
) 
\left(\twonorm{
    x^{\star} -x^{k}
    }^2+
    \frac{2}{n}\sum^{n}_{i=1}\twonorm{x^{k-i}-x^{k}}^2
    \right),
\end{align*}
where we plugged in \eqref{eq:shift-index}.
As in Lemma~\ref{lem:ybound1}, choose $\beta=1/n$ and we have
\begin{align*}
&2\sigma_{n-1}n^2\tau_1
(1 +  2\sigma_{n-1} n\tau_1
+
\sigma_{n-1}n^2\tau_1 
) \\
&\le
33c^2\exp(8c^2)
\left(
1+
\frac{33c^2\exp(8c^2)}{n}
+
16.5c^2\exp(8c^2)
\right)\\
&\le
33c^2\exp(16c^2)\left(
1+
16.5c^2
+
\frac{33c^2}{n}
\right)\\
&=
33
c^2(1+50c^2)\exp(16c^2).
\end{align*}
\end{proof}

\begin{lemma}\label{lem:bound_delta}
Define
\[
\Delta^k  = (x^{k+n} - x^k+ n\gamma \nabla f(x^k))/n\gamma.
\]
Let $\gamma=c/(L\sqrt{n(n+1)})$, where $c>0$.
Then
\[
    \twonorm{\Delta^k}^2
    \le
    50c^2\exp\left(8c^2\right)L^2
    \twonorm{
     x^k - x^\star
    }^2
    +
    \left(4 +
    200c^2\exp\left(8c^2\right)
    \right)\frac{L^2}{n}
\sum^{n}_{i=1}\twonorm{x^{k-i}-x^{k}}^2.
\]
\end{lemma}
\begin{proof}
Note that
\begin{align*}
&x^{k+n}-x^k=
\sum^{n-1}_{j=0}
x^{k+j+1}-x^{k+j}\\
&=-\gamma
\sum^{n-1}_{j=0}
\left(
\nabla f_{[k+j]}(x^{k+j})-\nabla f_{[k+j]}(x^{k+j-n})
+\frac{1}{n}\sum^{n}_{i=1}
\nabla f_{[k+j-i]}(x^{k+j-i})
\right)
\\
&=-\gamma
\sum^{n-1}_{j=0}
\left(
\nabla f_{[k+j]}(x^{k+j})
+\left(1-\frac{j+1}{n}\right)
\left(
\nabla f_{[k+j]}(x^{k+j})-\nabla f_{[k+j]}(x^{k+j-n})
\right)
\right).
\end{align*}
Define
\[
\Delta_j^k = \nabla f_{[k+j]}(x^{k+j})-\nabla f_{[k+j]}(x^{k})
+\left(1-\frac{j+1}{n}\right)
\left(
\nabla f_{[k+j]}(x^{k+j})-\nabla f_{[k+j]}(x^{k+j-n})\right).
\]
Then we have 
\[
\Delta^k=-\frac{1}{n}\sum^{n-1}_{j=0}\Delta^k_j.
\]
Using the fact that
$(1-(j+1)/n)\le 1$, we get
\begin{align*}
\twonorm{\Delta^k_j}^2
&\le
2L^2 \twonorm{x^{k+j}-x^k}^2
+
2L^2
\twonorm{x^{k+j}-x^{k+j-n}}^2\\
&\le
2L^2 \twonorm{x^{k+j}-x^k}^2
+
4L^2
\twonorm{x^{k+j}-x^k}^2
+4L^2\twonorm{x^{k+j-n}-x^{k}}^2\\
&=
6L^2 \twonorm{x^{k+j}-x^k}^2
+4L^2\twonorm{x^{k+j-n}-x^{k}}^2\\
&\le
\frac{j}{n}100c^2L^2\exp(8c^2)\twonorm{x^k-x^\star}^2
+200\frac{c^2L^2}{n}\exp(8c^2)\sum^n_{i=1}\twonorm{x^{k-i}-x^k}^2
+4L^2\twonorm{x^{k+j}-x^k}^2.
\end{align*}
So
\begin{align*}
&\|\Delta^k\|^2\le \frac{1}{n}\sum^n_{j=1}\|\Delta_j^k\|^2\\
&\le
50^2L^2\exp(8c^2)\twonorm{x^k-x^\star}^2
+(4+200c^2\exp(8c^2))\frac{L^2}{n}\sum^n_{i=1}\twonorm{x^{k-i}-x^k}^2,
\end{align*}
where we used \eqref{eq:triangle-like} and $\sum_{j=0}^{n-1} j < n^2/2$.
\qed
\end{proof}

\begin{proof}[Theorem~\ref{thm:main}]
We define
\[
x^{k+n}=x^k-n\gamma \nabla f(x^k)+n\gamma\Delta^k
\]
We use \eqref{eq:beta} and \eqref{eq:young} to get
\begin{align*}
\|x^{k+n}-x^\star\|^2
&=
\|x^{k}-x^\star\|^2
-2n\gamma\langle x^k-x^\star,\nabla f(x^k)\rangle
+2n\gamma \langle x^k-x^\star,\Delta^k\rangle\\
&\quad+n^2\gamma^2\|\nabla f(x^k)-\Delta^k\|^2\\
&\le 
\|x^{k}-x^\star\|^2
-2n\gamma\langle x^k-x^\star,\nabla f(x^k)\rangle
+\varepsilon n\gamma \| x^k-x^\star\|^2\\
&\quad+n\gamma/\varepsilon\|\Delta^k\|^2
+2n^2\gamma^2\|\nabla f(x^k)\|^2+2n^2\gamma^2\|\Delta^k\|^2,
\end{align*}
for any $\varepsilon>0$.
Apply \eqref{eq:lem:cocoercivity2}\ to get
\begin{align*}
\twonorm{x^{k+n} - x^\star}^2 
&\le (1-2\gamma n\frac{\mu L}{\mu+L} +  n\gamma \epsilon )\|x^{k}-x^\star\|^2 \\
&\quad-\frac{n\gamma}{L+\mu}
\left(
1-2n\gamma(L+\mu)
\right)
\|\nabla f(x^k)\|^2
+(n\gamma /\epsilon  + 2n^2\gamma^2 )\twonorm{\Delta^k}^2.
\end{align*}
Since $\kappa\ge 1$, we have $c<1/65$.
So $1-2n\gamma(L+\mu)<1$ and we have
\begin{align}
\twonorm{x^{k+n} - x^\star}^2 
&\le (1-2\gamma n\frac{\mu L}{\mu+L} +  \gamma \epsilon )\|x^{k}-x^\star\|^2 
+n\gamma(1 /\epsilon+2n\gamma) \twonorm{\Delta^k}^2 \nonumber
\\
&\le (1- \gamma n \mu +  n\gamma \epsilon )\|x^{k}-x^\star\|^2 
+n\gamma(1 /\epsilon+2n\gamma) \twonorm{\Delta^k}^2, \label{eq:Lp1_bound}
\end{align}
where the last inequality holds due to $\frac{\mu L}{\mu + L} = \frac{\mu}{\kappa^{-1} + 1} \geq \frac{\mu}{2}$.
Using the assumption $c < 1/65$, we simplify Lemma~\ref{lem:bound_delta} into
\begin{align}\label{eq:d_bound}
    \twonorm{\Delta^k}^2
    \le
    52 c^2 L^2
    \twonorm{
     x^k - x^\star
    }^2
    +
    5 \frac{L^2}{n}
    \sum_{j=1}^{n}
    \twonorm{
    x^k -x^{k-j}
    }^2     
\end{align}
and Lemma~\ref{lem:Lp2_bound} into
\begin{align}\label{eq:Lp2_bound}
    \twonorm{x^{k+n} - x^{k+j}}^2
    \leq
    35c^2
    \left(
    \twonorm{
     x^k - x^\star
    }^2
    +
    \frac{2}{n}\sum^n_{j=1}
    \|x^k-x^{k-j}\|^2
    \right).
\end{align}
We plug in \eqref{eq:Lp1_bound}, \eqref{eq:Lp2_bound} and \eqref{eq:d_bound} to have the following
\begin{align*}
V^{k+n} &=
\twonorm{x^{k+n} - x^\star}^2 + \frac{1}{n} \sum_{j=1}^n \twonorm{ x^{k+n} - x^{k+n-j} }^2  
\\&\leq 
\left(1-\gamma n\mu +  n\gamma \epsilon + (n\gamma /\epsilon + 2 n^2\gamma^2)  52 c^2 L^2 + 35c^2 \right)\|x^{k}-x^\star\|^2 
\\ & \quad +
\left( 5 (n\gamma /\epsilon + 2 n^2\gamma^2) L^2 + \frac{70c^2}{n} \right) \frac{1}{n}
\sum_{j=1}^n \twonorm{ x^{k} - x^{k-j} }^2  
\\&\leq 
\left(1-\gamma n\mu \left( 1- \frac{\epsilon}{\mu} - \frac{52 c^2 L^2}{\epsilon \mu} - \frac{104 c^2 n\gamma L^2 }{\mu} - \frac{50cL }{\mu} \right)  \right)\|x^{k}-x^\star\|^2 
\\ & \quad +
\left( 5 (n\gamma /\epsilon + 2 n^2\gamma^2) L^2 + \frac{70c^2}{n} \right) \frac{1}{n}\sum_{j=1}^n \twonorm{ x^{k} - x^{k-j} }^2  
\\& \leq
\left(1-\gamma n\mu \left( 1- \frac{\epsilon}{\mu} - \frac{52 c^2 L \kappa}{\epsilon} - 104 c^3 \kappa   -  50 c \kappa \right) \right)  \|x^{k}-x^\star\|^2 
\\ & \quad +
\left(5 (cL/ \epsilon + 2 c^2) + \frac{70c^2}{n} \right) \frac{1}{n}  \sum_{j=1}^n \twonorm{ x^{k} - x^{k-j} }^2, 
\end{align*}
where the second and third inequality holds due to the fact that $\sqrt{n(n+1)}\le \sqrt{2}n$ and  $n\gamma = \frac{\sqrt{n}}{\sqrt{n+1}L}\leq \frac{c}{L}$.

% \textbf{Case 1: $\gamma = \frac{c}{nL\kappa}$, $c < 1/18$, $D = 1/n$}\\
For the choice of $\epsilon = \sqrt{52}c L$, since 
\begin{align*}
\frac{\epsilon}{\mu} + \frac{52c^2L^2}{\epsilon \mu} = 2\sqrt{52} c \kappa
% ,\quad \text{and} \quad
% \frac{\gamma nL^2}{\epsilon} = \frac{1}{\sqrt{51n(n+1)} }
\end{align*}
holds, the inequality above becomes 
\begin{align*}
V^{k+n}
&\leq 
\left(1-\gamma n\mu ( 1-  2\sqrt{52} c \kappa( 1 + \sqrt{52}c^2 + 3.5 ) )\right)  \|x^{k}-x^\star\|^2 \\
&\qquad+ 
\left(\frac{5}{\sqrt{52}}+10c^2+\frac{70c^2}{n} \right) \frac{1}{n}  \sum_{j=1}^n \twonorm{ x^{k} - x^{k-j} }^2 .
\end{align*}
Since $\kappa\ge 1$, 
\begin{align*}
V^{k+n}
&\leq 
\left(1-\gamma n\mu ( 1-  2\sqrt{52} c \kappa( 1 + \sqrt{52}c^2\kappa^2 + 3.5 ) )\right)  \|x^{k}-x^\star\|^2 \\
&\qquad
+ 
\left(\frac{5}{\sqrt{52}}+10c^2\kappa^2+\frac{70c^2\kappa^2}{n} \right) \frac{1}{n}  \sum_{j=1}^n \twonorm{ x^{k} - x^{k-j} }^2 .
\end{align*}
Under $c\kappa <1/65$, it is simple to check
\[
2\sqrt{52} c \kappa( 1 + \sqrt{52}c^2\kappa^2 + 3.5 ) < 1, 
\]
and 
\[    
\frac{5}{\sqrt{52}}+10c^2\kappa^2+\frac{70c^2\kappa^2}{n}<1.
    % \frac{5}{\sqrt{52}} + 10c^2 + \frac{12c^2}{n} < 1.
\]
This gives us a contraction.
For $c\kappa = 1/130$, the contraction factor is 
\begin{align*}
\max\left\{1-\frac{\gamma n \mu}{2}
,
\frac{5}{\sqrt{52}}+
10c^2\kappa^2
+
\frac{10c^2\kappa^2}{n}
\right\}
%\le\max\left\{1-
%\frac{1}{368\kappa^2}
%,
%\frac{1}{\sqrt{2}}
%\right\}
\le
1-
\frac{1}{368\kappa^2},
\end{align*}
where again we use the fact that $\sqrt{n(n+1)}\le \sqrt{2}n$ and $\kappa \ge 1$.
\qed\end{proof}

\section{Experiments}
\begin{figure}
    \centering
    \includegraphics[scale=0.38]{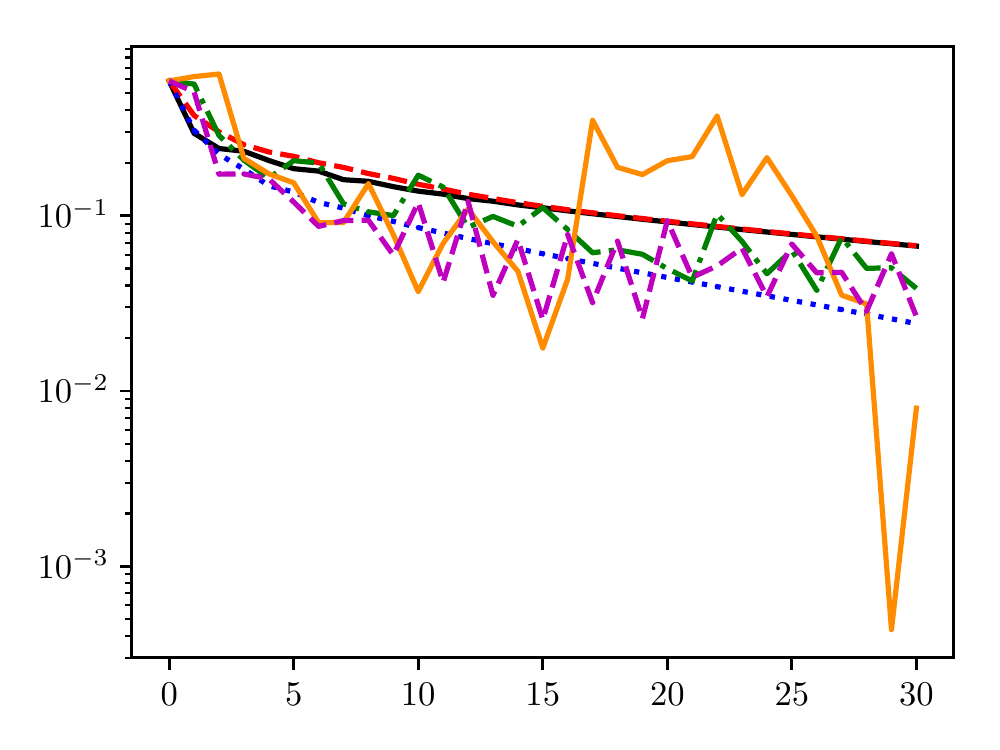}
    \includegraphics[scale=0.38]{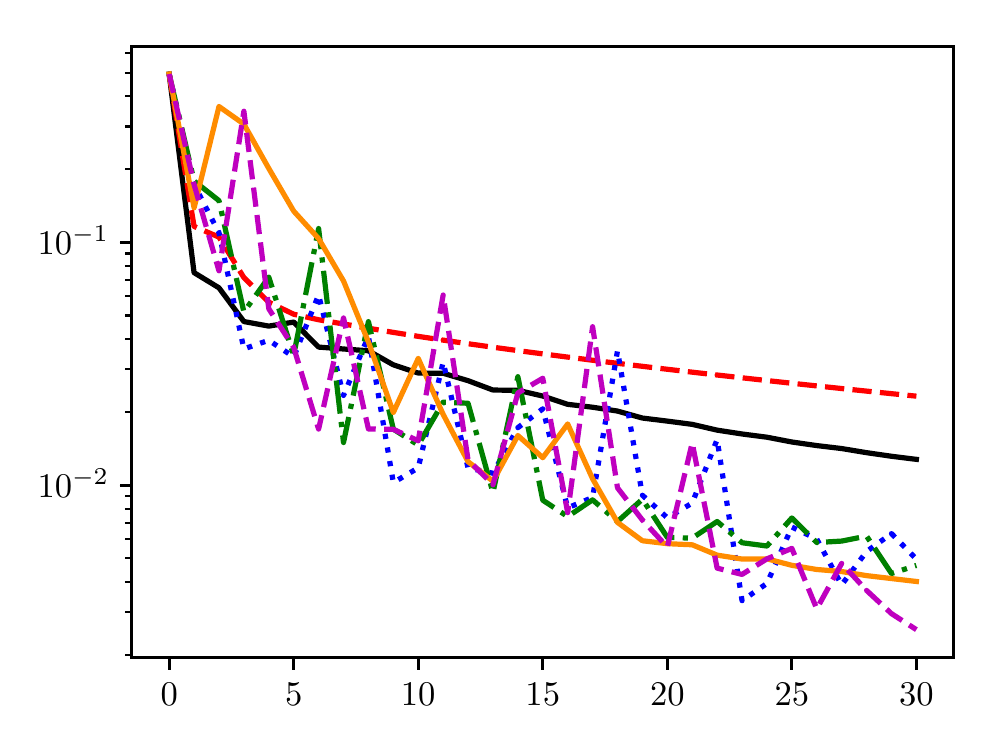}
    \includegraphics[scale=0.38]{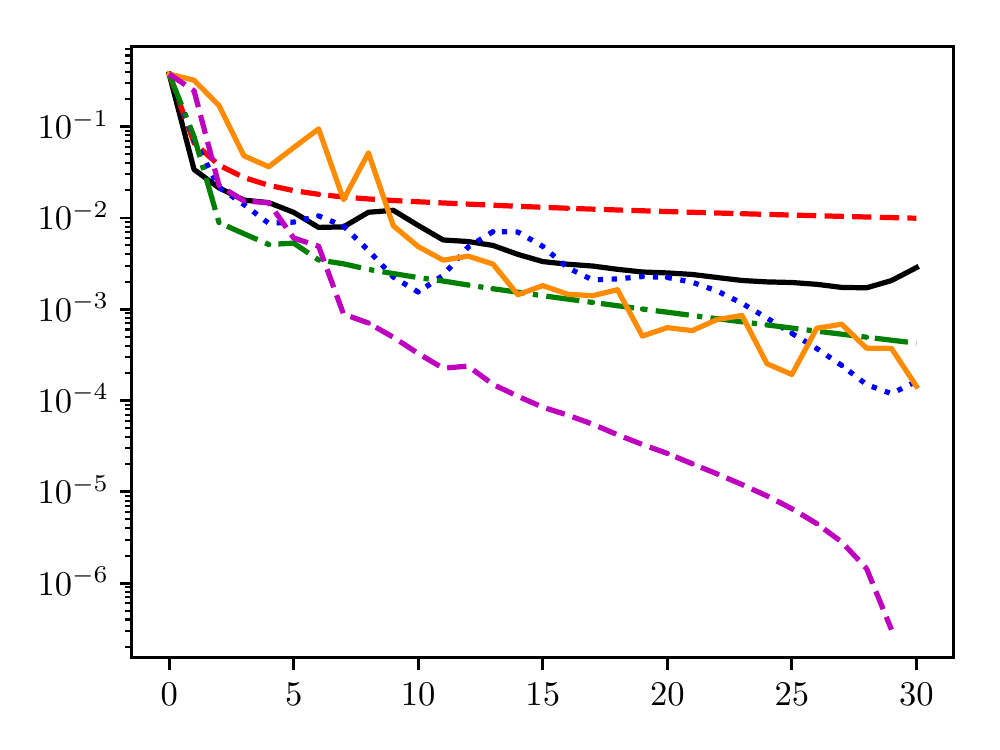}
    % \caption{the results for $5$, $10$, $100$ percent of australian dataset}
    % \label{fig:australian}
% \end{figure}
% \begin{figure}
\\
    \centering
    \includegraphics[scale=0.38]{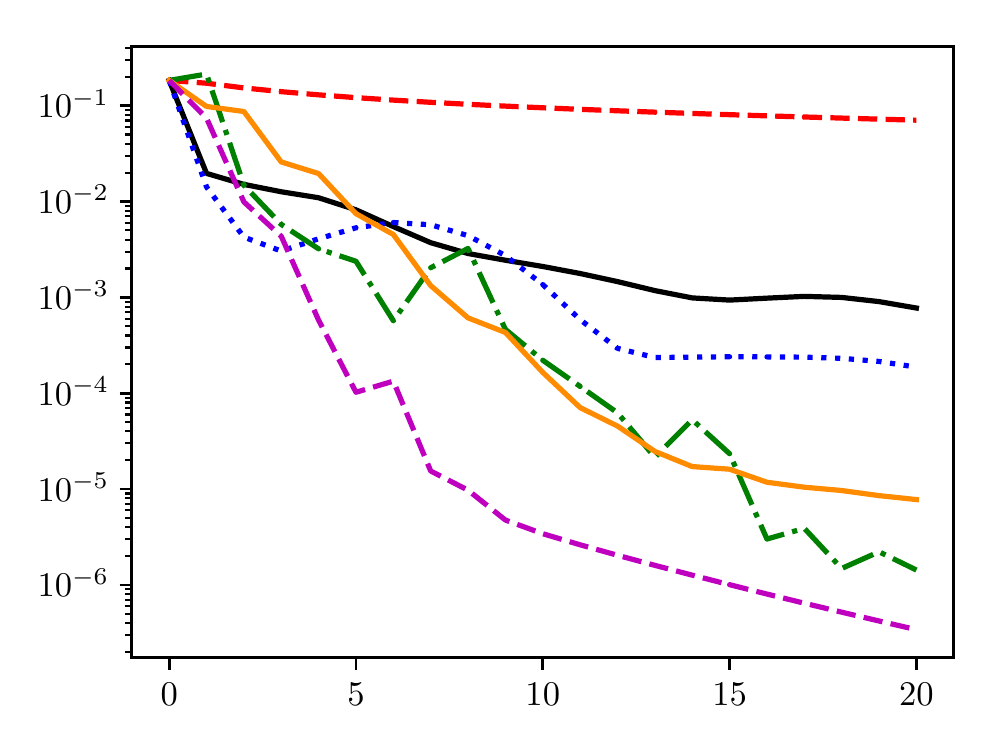}
    \includegraphics[scale=0.38]{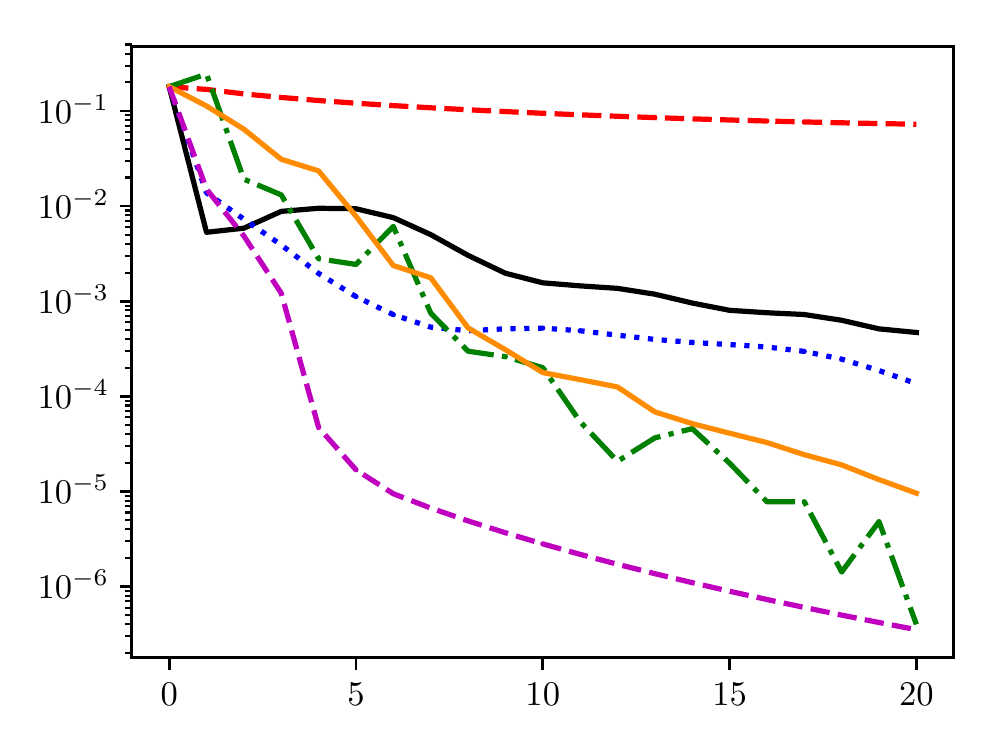}
    \includegraphics[scale=0.38]{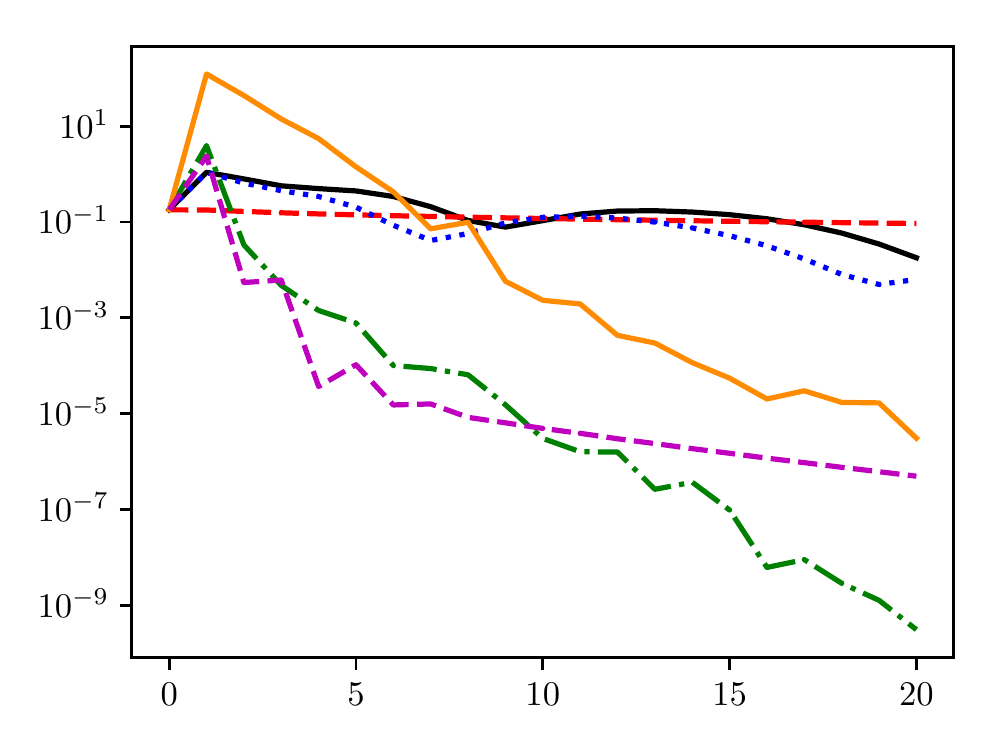}
    % \caption{the results for $5$, $10$, $100$ percent of covtype dataset}
    % \label{fig:covtype}
% \end{figure}
% \begin{figure}
\\
    \centering
    \includegraphics[scale=0.38]{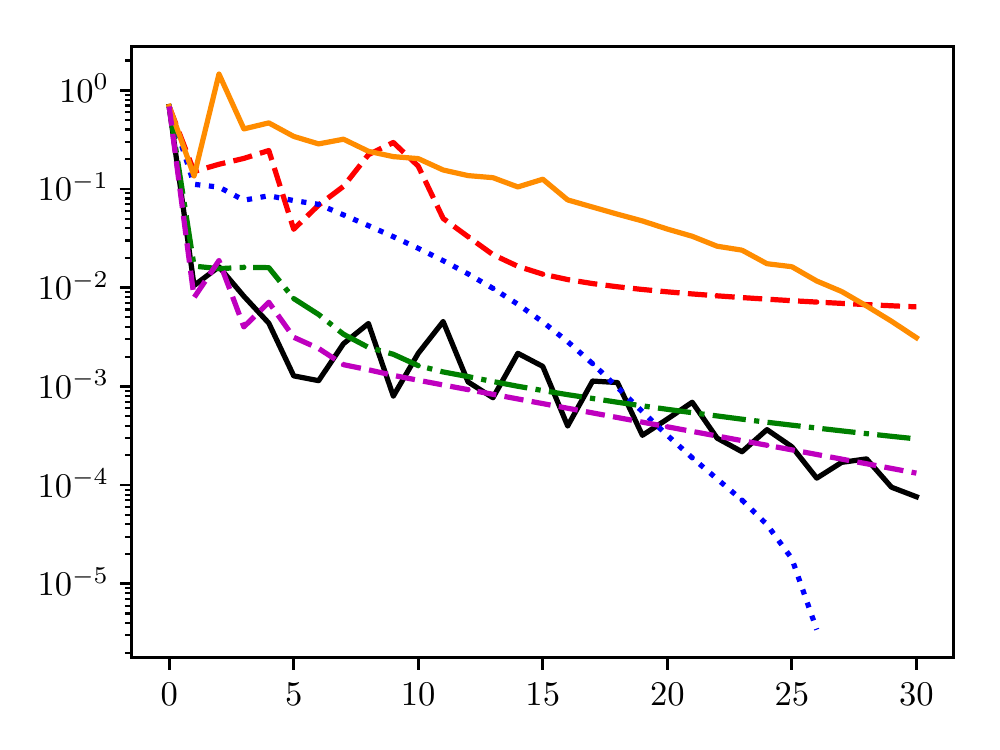}
    \includegraphics[scale=0.38]{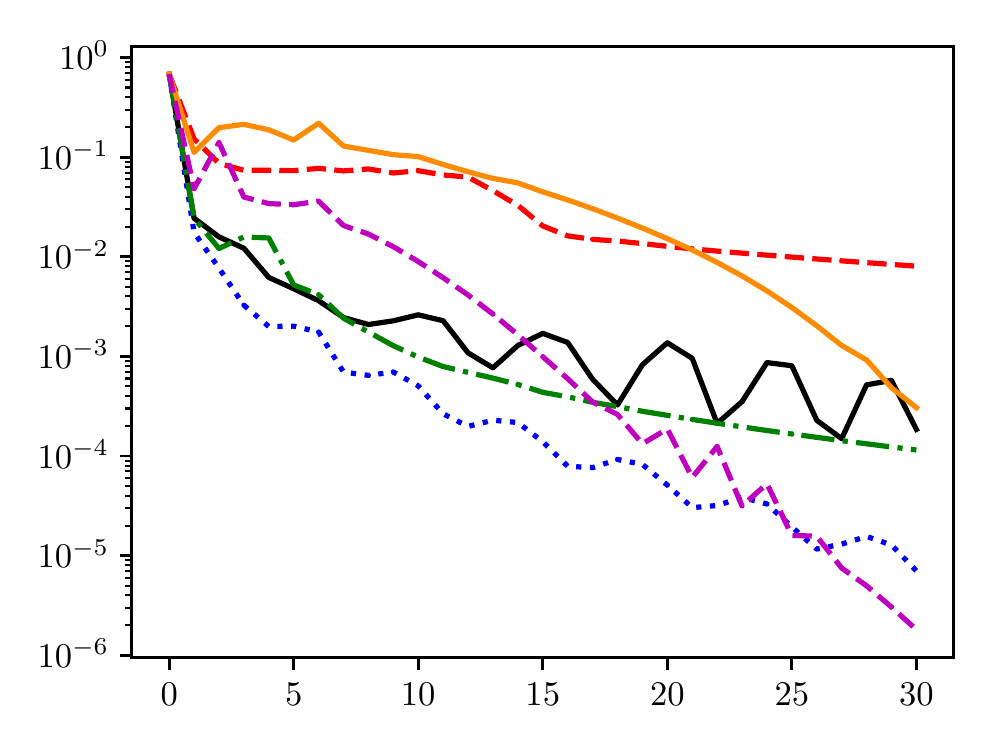}
    \includegraphics[scale=0.38]{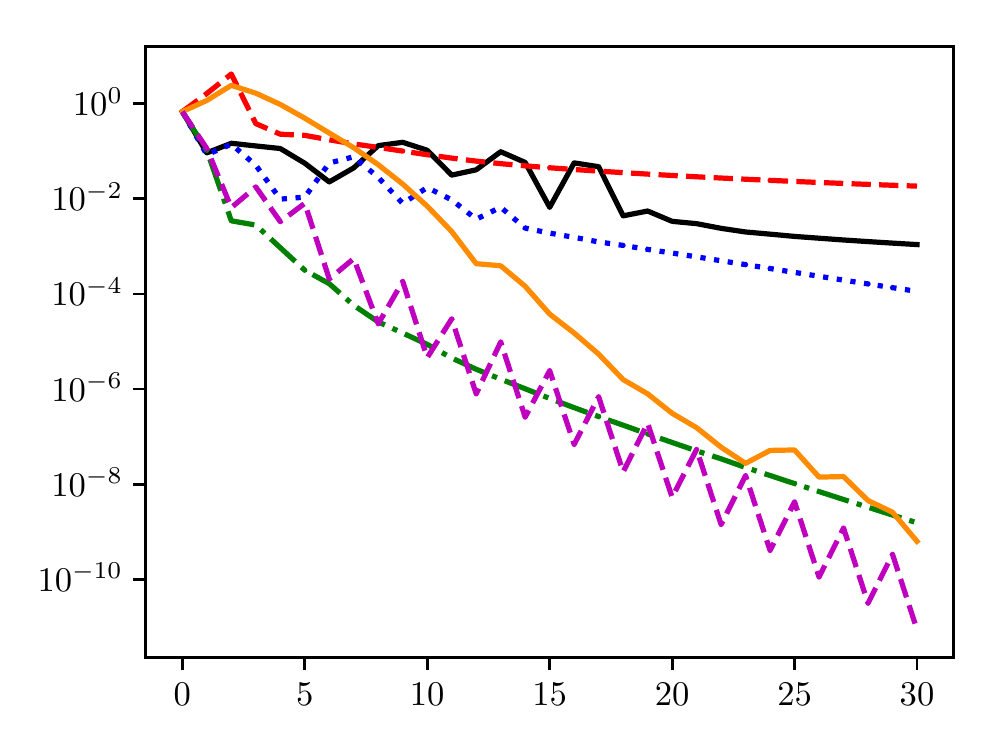}
%     \caption{the results for $5$, $10$, $100$ percent of mushrooms dataset}
%     \label{fig:mushrooms}
% \end{figure}
% \begin{figure}
\\
    \centering
    \includegraphics[scale=0.38]{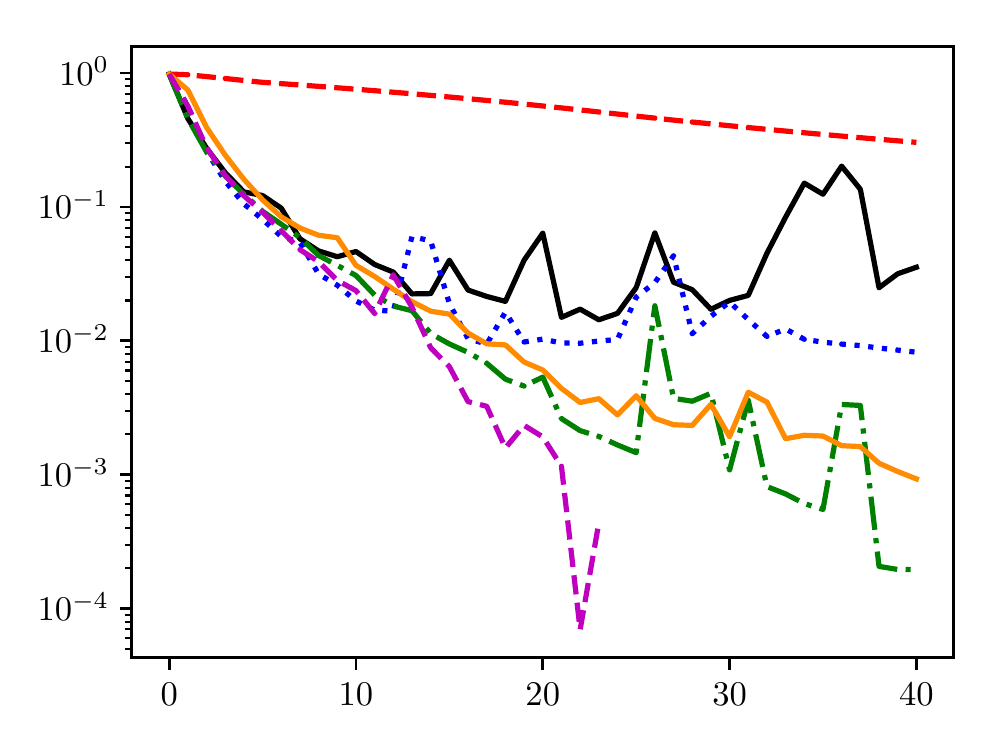}
    \includegraphics[scale=0.38]{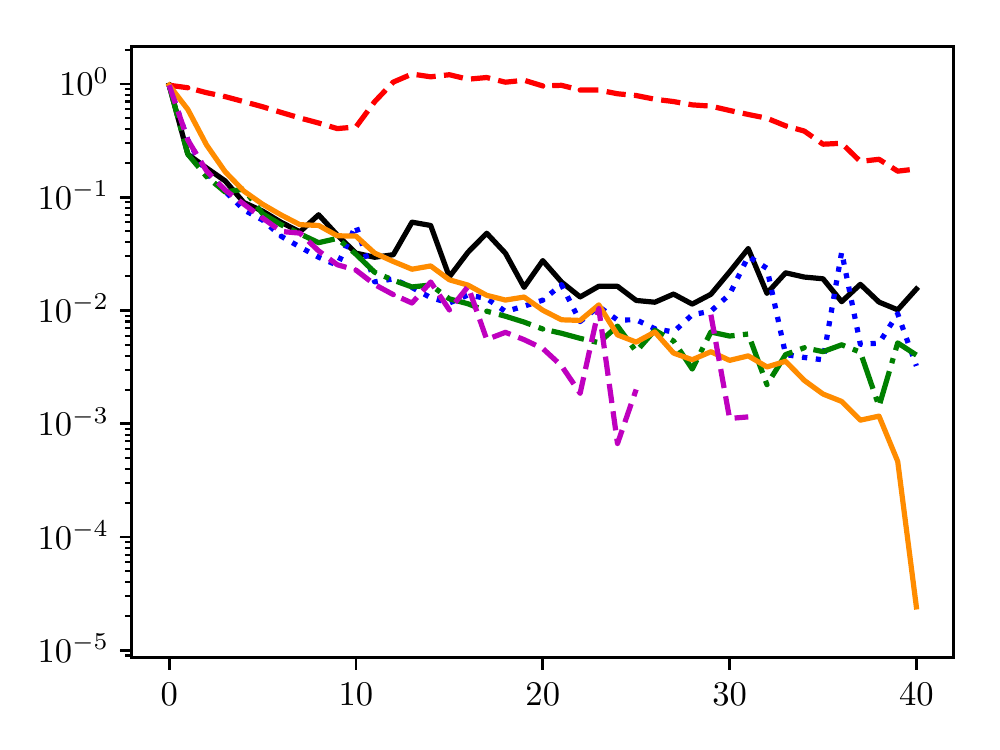}
    \includegraphics[scale=0.38]{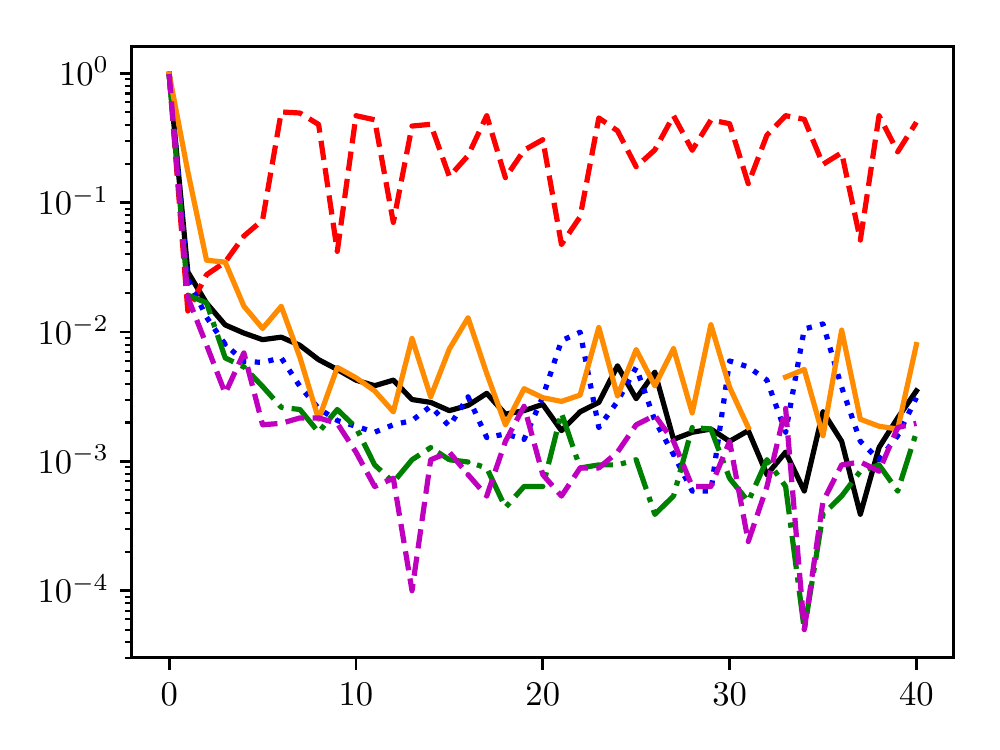}
\\
    \centering
    \includegraphics[scale=0.33]{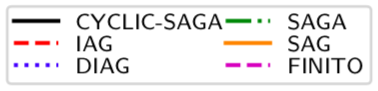}
    \caption{Function suboptimality vs.\ epoch: from top to bottom, the rows corresponds to the AUTRALIAN, COVTYPE, MUSHROOM, and RCV1 dataset. From left to right, each column corresponds to a $5$, $10$, $100$ percent subsampling of the dataset. 
    For the bottom right corner setup (the entire RCV1 dataset) Finito and DIAG took more than $10$ hours, while the other four methods took less than $5$ seconds due to just-in-time updates.}
    \label{fig:rcv1}
\end{figure}

Figure~\ref{fig:rcv1} shows the experimental results.
Overall, we observe that C-SAGA performs better than IAG, which is consistent with the theory.
We also observe that C-SAGA is slower than DIAG in iteration count but is faster in wall clock time due to the acceleration just-in-time updates provide. 
For the RCV dataset with $n=20242$ and $m=47237$, DIAG and Finito (the randomized version of DIAG) took more than $10$ hours while the other methods took less than $5$ seconds. 
% We additionally observe that for each algorithm, the randomized version is better than the deterministic one. 
% However, there is no clear winner 
% among three different types of algorithms. 
% In some cases, the deterministic version of one algorithm (C-SAGA) 
%  outperforms the randomized version of the other algorithm (SAG).
For the experiments, we modified Defazio's code \cite{defazio_pointsaga}, which implements SAG and SAGA, but not Finito.
The dataset is a selection of commonly used datasets from the LIBSVM repository \cite{libsvm}.
%and applied cyclic, IID random, and random permutation selection rules.
For IAG, C-SAGA, SAG, and SAGA, we use just-in-time updates (implemented by Defazio) to accelerate the computation as explained in Section 4.1 of \cite{schmidt2016}.
As discussed in Section~\ref{s:proof}, just-in-time updates are not applicable to DIAG and Finito.
For each algorithm, we ran experiments for a wide range of stepsizes from $8192$ to $10^{-4}$ and chose the best one.
The time measurements were taken on a MacBook Air with a 1.3 GHz Intel Core i5 CPU.

\section{Conclusion}
In this work, we analyzed C-SAGA and compared C-SAGA to existing methods, theoretically and experimentally.
As an aside, we observed that the random permutation variant of SAGA outperforms SAGA and C-SAGA. Our experimental results conform with the curiously persistent phenomenon that, among index selection rules, random permutation outperforms IID, which in turn outperforms cyclic. 
Investigating the effect of random permutations on incremental methods and systematically comparing the deterministic and randomized complexity of finite-sum optimization problems is an interesting direction of future work.

\bibliographystyle{spmpsci}      % basic style, author-year citations
\bibliography{./refs}

\end{document}